\documentclass[12pt,a4paper]{amsart}
\usepackage{amsmath,amsthm}
\usepackage{amssymb,latexsym}
\usepackage{enumerate}
\usepackage{amsmath,amssymb,color,inputenc,euscript,graphicx}
\usepackage{tikz}
\usepackage{float}

\theoremstyle{plain}
\newtheorem{theo}{Theorem}[section]
\newtheorem{theorem}[theo]{Theorem}
\newtheorem{theo*}{Theorem}

\newtheorem{proposition}[theo]{Proposition}
\newtheorem{lem}[theo]{Lemma}
\newtheorem{lemma}[theo]{Lemma}

\theoremstyle{remark}

\newtheorem{remark}[theo]{Remark}

\newcommand{\reff}[1]{(\ref{#1})}
\newcommand{\ind}{{\bf 1}}

\newcommand{\ca}{\mathcal A}
\newcommand{\cl}{\mathcal L}
\newcommand{\cm}{\mathcal M}
\newcommand{\cu}{\mathcal U}

\newcommand{\bt}{\mathbf{t}}

\newcommand{\N}{\mathbb{N}}

\renewcommand{\P}{\mathbb{P}}
\newcommand{\T}{\mathbb{T}}
\newcommand{\E}{\mathbb{E}}

\newcommand{\Card}{\mathrm{Card}}

\newcommand{\anc}{\mathrm{An}}
\newcommand{\dist}{\mathrm{dist}}

\title[Protected nodes]{Local limits of Galton-Watson trees
  conditioned on the number of protected nodes}

\author{Romain Abraham}
\address{Romain Abraham, Laboratoire MAPMO, CNRS, UMR 7349, F\'{e}d\'{e}ration Denis
  Poisson, FR 2964, Universit\'{e} d'Orl\'{e}ans, B.P. 6759, 45067
  Orl\'{e}ans cedex 2, France.}

\email{romain.abraham@univ-orleans.fr}

\author{Aymen Bouaziz}
\address{Aymen Bouaziz, Institut pr\'{e}paratoire aux \'{e}tudes scientifiques et
  techniques, 2070 La Marsa,  Tunisie.}
\email{bouazizaymen16@yahoo.com}

\author{Jean-Fran\c cois Delmas} 
\address{Jean-Fran\c cois Delmas, CERMICS, Ecole des Ponts,
  UPE, Champs-sur-Marne, France.}
\email{delmas@cermics.enpc.fr}

\begin{document}

\begin{abstract}
  We consider a  marking procedure of the vertices of  a tree where each
  vertex is marked independently from the others with a probability that
  depends only on its out-degree. We prove that a critical Galton-Watson
  tree conditioned on having a large number of marked vertices converges
  in distribution to the associated size-biased tree. We then apply this
  result to give  the limit in distribution of a  critical Galton-Watson
  tree conditioned on having a large number of protected nodes.
\end{abstract}



\date{\today}
 
\maketitle

\section{Introduction}

In \cite{k86}, Kesten proved that a critical or sub-critical
Galton-Watson (GW) tree conditioned on reaching at least height $h$
converges in distribution (for the local topology on trees) as
$h$ goes to infinity toward the
so-called sized-biased tree (that we call here Kesten's tree and whose
distribution is described in Section \ref{sec:Kesten}). Since then,
other conditionings have been considered, see \cite{ad14,ad15,h15} and
the references therein for recent developments on the subject.

A protected node is a node that is not a leaf and none of its offsprings
is a  leaf. Precise asymptotics for  the number of protected  nodes in a
conditioned GW  tree have already  been obtained in  \cite{dj14,j15} for
instance. Let  $A(\bt)$ be  the number  of protected  nodes in  the tree
$\bt$.  Remark that this functional $A$  is clearly monotone in the sense of
\cite{h15} (using  for instance \reff{eq:additivity});  therefore, using
Theorem 2.1  of \cite{h15}, we immediately  get that a critical  GW tree
$\tau$ conditioned  on $\{A(\tau)>n\}$ converges in  distribution toward
Kesten's tree as $n$ goes  to infinity.  Conditioning on $\{A(\tau)=n\}$
needs extra  work and is  the main objective  of this paper.   Using the
general result of \cite{ad14}, if we have the following limit
\begin{equation}
\label{eq:lim}
\lim_{n\to+\infty}\frac{\P(A(\tau)=n+1)}{\P(A(\tau)=n)}=1,
\end{equation}
then  the  critical  GW  tree  $\tau$  conditioned  on  $\{A(\tau)=n\}$  converges  in
distribution also toward Kesten's tree,  see   Theorem
\ref{thm:protected}.

In fact, the limit \reff{eq:lim} can be seen as a special case of a more
general problem: conditionally given the tree,  we mark the nodes of the
tree  independently of  the rest  of the  tree with  a probability  that
depends only on the number of offsprings  of the nodes. Then we prove 
that a critical GW tree  conditioned on the total number of marked nodes
being large converges in  distribution toward
Kesten's tree, see Theorem  \ref{thm:main}.

The paper is then organized as follows: we first recall briefly the
framework of discrete trees, then we consider in Section
\ref{sec:marked} the problem of a marked GW tree and the
proofs of the results are given in Section \ref{sec:proof}. In
particular, we prove the limit \reff{eq:lim} when $A$ is the number of
marked nodes in Lemma \ref{lem:ratio} and we deduce the convergence of
a critical GW tree conditioned on the number of marked
nodes toward Kesten's tree in Theorem \ref{thm:main}. We finally
explain in Section \ref{sec:protected} how the problem of protected
nodes can be viewed as a problem on marked nodes and deduce the
convergence in distribution of a critical GW tree
conditioned on the number of protected nodes toward Kesten's tree in
Theorem \ref{thm:protected}.

\section{Technical background on GW trees}

\subsection{The set of discrete trees}\

We denote by $\N=\{0,1,2,\ldots\}$ the set of non-negative integers
and by
$\N^{*}=\{1,2,\ldots\}$ the set of positive integers.

If $E$ is a subset  of $\mathbb{N}^{*}$,  we call the span of $E$ the 
greatest common divisor of $E$. If $X$ is an  integer-valued   random 
variable, we call the span of $X$ the span of $\{n>0;\ \P(X=n)>0\}$.

We recall Neveu's formalism \cite{n86} for ordered rooted trees. Let
$\mathcal{U}=\bigcup_{n\geq 0}(\mathbb{N}^{*})^{n}$ 
be the set of finite sequences  of positive integers with the convention
$(\mathbb{N}^{*})^{0}=\{\emptyset\}$.  For $u\in \mathcal{U}$, its
length or generation $|u|\in \N$ is defined by 
$u  \in(\mathbb{N}^{*})^{|u|}$.   If $u$  and  $v$  are two  sequences  of
$\mathcal{U}$, we denote by $uv$ the concatenation of the two sequences,
with  the  convention  that  $uv=u$   if  $v=\emptyset$  and  $uv=v$  if
$u=\emptyset$.  The set of ancestors of $u$ is the set
\[
\anc(u)=\{v \in \mathcal{U};\ \exists w \in \mathcal{U} \text{ such that
} u=vw\}.
\]
Notice that $u$ belongs to $\anc(u)$. For two distinct
elements $u$ and $ v$ of $ \mathcal{U}$, we denote  by $u<v$ the lexicographic  order on
$\mathcal{U}$ i.e.  $u<v$  if $u \in \anc(v)$ and $u\neq v$ or if  $u=wiu'$ and $v=wjv'$
for some $i,j \in \mathbb{N}^{*}$ with $i<j$. We write $u\leq v$ if
$u=v$ or $u<v$. 

A tree $\bt$ is a subset of $\mathcal{U}$ that satisfies: 
\begin{itemize}
\item  $\emptyset \in \bt$.
\item  If $u\in \bt$, then $\anc(u)\subset \bt$.
\item  For every $u\in \bt$, there exists 
  $k_{u}(\bt)\in\N$ such that, for every $i\in\N^*$, $ui \in \bt$ iff
  $1\leq i\leq k_{u}(\bt)$.
\end{itemize}

The vertex $\emptyset$ is called  the root of $\bt$.
The integer $k_{u}(\bt)$ represents the number of offsprings of the
vertex $u \in \bt$. The set of children of a vertex $u\in \bt$ is given
by:
\begin{equation}
   \label{eq:children}
C_u(\bt)= \{ui ;\  1\leq i\leq k_u(\bt)\}.
\end{equation}
By convention, we set $k_u(\bt)=-1$ if $u \not\in
\bt$. 

A vertex $u \in  \bt$ is called a leaf if  $k_{u}(\bt)=0$.  We denote by
$\mathcal{L}_0(\bt)$ the set of leaves of $\bt$. A vertex $u \in \bt$ is
called  a   {\sl  protected   node}  if  $C_u(\bt)\neq   \emptyset$  and
$C_u(\bt)\bigcap \cl_0(\bt)=\emptyset$,  that is $u$  is not a  leaf and
none of its  children is a leaf. For $u\in \bt$, we  define
$F_{u}(\bt)$, the fringe subtree  of $\bt$ above $u$, as
$$F_{u}(\bt)=\{v \in \bt;\  u \in \anc(v)\}=\{uv;\  v \in S_{u}(\bt)\}$$
with $S_{u}(\bt)=\{v \in \mathcal{U}; uv \in \bt\}$.

Notice that $ S_{u}(\bt)$ is a tree. We  denote  by   $\mathbb{T}$ the set of
trees and by
$\mathbb{T}_{0}=\{\bt  \in  \mathbb{T};\  \mathrm{Card}(\bt)<+\infty \}$
the subset of finite trees.

We say that a sequence of  trees $(\bt_n,\  n\in \N)$ converges locally to
a tree $\bt$  if and  only if  $\lim_{n\rightarrow \infty  } k_u(\bt_n)
=k_u(\bt)$  for all  $u\in \cu$.  Let  $(T_{n},\ n  \in \N)$  and $T$  be
$\mathbb{T}$-valued  random  variables.   We denote  by  $\dist(T)$  the
distribution of the random variable $T$ and
write 
\[
\lim_{n\longrightarrow +\infty}\dist(T_{n})= \dist(T)
\]
for the convergence in distribution of  the sequence $(T_{n},\  n \in \N)$
to $T$ with respect to the local topology.

If $\bt,\bt' \in \mathbb{T}$ and $x \in \mathcal{L}_{0}(\bt)$ we
denote by
\begin{equation}\label{eq:graft}
\bt\circledast_x\bt'=\{u \in \bt\}\cup\{xv;\ v\in \bt'\}
\end{equation}
the tree obtained by grafting the tree $\bt'$ on the leaf $x$ of the tree $\bt$.
For every $\bt \in \mathbb{T}$ and every $x \in \mathcal{L}_0(\bt)$, we shall
consider the set of trees obtained by grafting a tree on the leaf $x$ of $\bt$:
$$\mathbb{T}(\bt,x)=\{\bt\circledast_x\bt';\  \bt' \in \mathbb{T}\}.$$

\subsection{Galton Watson trees}

Let  $p=(p(n),\ n  \in  \mathbb{N})$ be  a  probability distribution  on
$\N$. We assume that
\begin{equation}\label{eq:assumption}
    p(0)>0 ,\ p(0)+p(1)<1 ,\mbox{ and } \mu :=\sum_{n=0}^{+\infty}np(n) <+\infty.
\end{equation}

A $\mathbb{T}$-valued random variable $\tau$ is a GW tree with offspring 
distribution $p$ if the distribution of $k_{\emptyset}(\tau)$ is $p$ and 
it enjoys the branching property: for $n \in \mathbb{N}^{*}$, conditionally 
on $\{k_{\emptyset}(\tau)=n\}$, the subtrees $(F_{1}(\tau),\ldots,F_{n}(\tau))$ 
are independent and distributed as the original tree $\tau$. 

The  GW  tree  and  the   offspring  distribution  are  called  critical 
(resp. sub-critical,  super-critical) if $\mu  =1$ (resp. $\mu  <1,\ \mu >1$).

\section{Conditioning on the number of marked vertices}
\label{sec:marked}

\subsection{Definition of the marking procedure}

We begin with a fixed tree $\bt$. We add marks on the vertices of $\bt$ in an independent way such that
the probability of adding a mark on a node $u$ depends only on the number of
children of $u$. More precisely, we consider a mark function
$q:\mathbb{N}\longrightarrow [0,1]$ and a family of independent
Bernoulli random
variables $(Z_u(\bt),\ u\in\bt)$ such that for all $u\in \bt$:
\[
 \P(Z_u(\bt)=1)=1 - \P(Z_u(\bt)=0)=q(k_u(\bt)).
\]

The  vertex $u$  is said  to  have a mark  if $Z_u(\bt)=1$.  We denote  by
$\cm(\bt)=\{u\in \bt; \, Z_u(\bt)=1\}$ the set of marked vertices and by
$M(\bt)$ its cardinal. We call $(\bt, \cm(\bt))$ a marked tree.

A marked GW tree with offspring  distribution $p$ and mark function $q$
is a couple  $(\tau, \cm(\tau))$, with $\tau$ a GW  tree with offspring
distribution $p$ and  conditionally on $\{\tau=\bt\}$ the  set of marked
vertices $\cm(\tau)$ is distributed as $\cm(\bt)$.

\begin{remark}\label{rem:qdeterministe}
Notice that for $\ca\subseteq \mathbb{N}$, if we set
$q(k)=\ind_{\{k\in\ca\}}$, then 
the set $\mathcal{M}(\bt)$ is just the set of vertices with
out-degree (i.e. number of offsprings) in $\ca$ considered in
\cite{ad14,r15}. Hence, the above construction can be seen as an
extension of this case.
\end{remark}

\subsection{Kesten's tree}
\label{sec:Kesten}

Let $p$ be an offspring distribution satisfying Assumption \reff{eq:assumption} with
$\mu \leq 1$ (i.e. the associated GW process is critical or
sub-critical). We denote by
$ p^{*}=(p^{*}(n)=np(n)/\mu,\ n\in \mathbb{N})$
the corresponding size-biased distribution.

We define an infinite random tree $\tau ^{*}$ (the size-biased tree
that we call Kesten's tree in this paper) whose distribution is
described as
follows: 

There exists a unique infinite sequence $(v_{k},\ k \in \mathbb{N}^{*})$
of positive integers such that, for every $h \in
\mathbb{N},\ v_{1}\cdots v_{h} \in \tau^{*}$, with the convention that
$v_{1}\cdots v_{h}=\emptyset$ if $h=0$. The joint distribution of
$(v_{k},\ k \in \mathbb{N}^{*})$ and $\tau^{*}$ is determined recursively as follows.
For each $h \in \mathbb{N}$, conditionally given $(v_{1},\ldots,v_{h})$
and $\{u\in \tau^*;\, |u|\leq h\}$ the tree $\tau^*$ up to level $h$, we have:
\begin{itemize}
\item  The number of children $(k_{u}(\tau^{*}),\ u \in
    \tau^{*},\ |u|=h)$ are independent and distributed according to $p$
    if $u\neq v_{1}\cdots v_{h}$ and according to $p^{*}$ if
    $u=v_{1}\ldots v_{h}$.
  \item Given  $\{u\in \tau^*;\, |u|\leq h+1\}$ and  $(v_1, \ldots, v_h)$,  the integer
    $v_{h+1}$  is   uniformly  distributed   on  the  set   of  integers
    $\{1,\ldots,k_{v_{1}\cdots v_{h}}(\tau^{*})\}$.
\end{itemize}

\begin{remark}\

  Notice   that   by  construction,   a.s.   $\tau^{*}   $  has a
  unique infinite spine.  And following  Kesten \cite{k86},  the random  tree
  $\tau^{*}$  can  be viewed  as  the  tree  $\tau$ conditioned  on  non
  extinction.
 
For $\bt \in \mathbb{T}_{0}$ and $x \in \mathcal{L}_0(\bt)$, we have:
\[
\P(\tau^{*}\in \mathbb{T}(t,x))=\frac{\P(\tau=t)}{\mu^{\vert x
    \vert}p(0)}\cdot
\]
\end{remark}

\subsection{Main theorem}

\begin{theorem}\label{thm:main}
  Let $p$ be a critical offspring distribution that satisfies Assumption
  \reff{eq:assumption}.   Let $(\tau,  \cm(\tau))$ be  a marked  GW tree
  with  offspring  distribution $p$  and  mark  function $q$  such  that
  $p(k)q(k)>0$ for some  $k\in \N$.  For every  $n\in\N^*$, let $\tau_n$
  be a tree whose distribution is the conditional distribution of $\tau$
  given  $\{M(\tau)=n\}$.  Let  $\tau^*$ be  a Kesten's  tree associated
  with $p$. Then we have:
\[
\lim_{n\to+\infty}\dist(\tau_n)=\dist(\tau^*),
\]
where  the limit  has to  be understood  along a  subsequence for  which
$\P(M(\tau)=n)>0$.
\end{theorem}

\begin{remark}
If for every $k\in\N$, $0<q(k)<1$, then
$\P(M(\tau))=n)>0$ for every $n\in\N$, hence the above
conditioning is always valid.
\end{remark}

\section{Proof of Theorem \ref{thm:main}}\label{sec:proof}
Set  $\gamma=\P(M(\tau)>0)$.    Since  there   exists  $k\in   \N$  with
$p(k)q(k)>0$, we  have $\gamma>0$.    A sufficient  condition (but
not necessary) to  have $\P(M(\tau)=n)>0$ for every $n$  large enough is
to  assume  that  $\gamma<1$  (see  Lemma  \ref{lem:LY=ok}  and  Section
\ref{sec:p-ratio}).  Taking $q=\ind_\ca$, see Remark
\ref{rem:qdeterministe}   for   $0 \in \ca \subset \N$   implies
$\gamma=1$ and some periodicity may occur.

The following result is the analogue in the random case of Theorem 3.1
in \cite{ad14} and its proof is in fact a straighforward adaptation of
the proof in \cite{ad14} by using: 
\begin{itemize}
\item [(i)] $M(\bt)\leq  \Card   (\bt)$.
\item[(ii)] For  every
$\bt\in\T_0$,  $x\in\mathcal{L}_0(\bt)$  and $\bt'\in\T$, we
have   that   $M(\bt\circledast_x\bt')$    is   distributed   as   $\hat
M(\bt')+M(\bt)-\ind_{\{Z_x(\bt)=1\}}$,    where   $\hat    M(\bt')$   is
distributed as $M(\bt')$ and is  independent of $\cm(\bt)$.
\end{itemize}

\begin{proposition}\label{prop:ratio}
Let $n_{0} \in \mathbb{N}\cup\{\infty\}$. Assume that $\P(M(\tau)\in [n,n+n_{0}))>0$ for $n$ large enough. Then, if 
\begin{equation}\label{eq:limit}
\lim_{n\longrightarrow+\infty} \frac{\P(M(\tau)\in
  [n+1,n+1+n_{0})}{\P(M(\tau) \in [n,n+n_{0}))}=1,
\end{equation}
we have: 
$$\lim_{n\longrightarrow+\infty} dist(\tau \vert M(\tau) \in [n,n+n_{0}))=dist(\tau^{*}).$$
\end{proposition}
 
\begin{proof}
According to Lemma 2.1 in \cite{ad14}, a sequence $(T_n, \,n\in \N)$ of
finite random trees converges in distribution (with respect to the
local topology) to some Kesten's tree $\tau^*$ if and only if, for
every finite tree $\bt\in \mathbb{T}_{0}$ and every leaf $x \in
\mathcal{L}_0(\bt)$,
\begin{equation}\label{eq:conv}
\lim_{n\to+\infty}\P\bigl((T_n\in
\mathbb{T}(\bt,x)\bigr)=\P\bigl(\tau^*\in
\mathbb{T}(\bt,x)\bigr)\quad\mbox{and}\quad
\lim_{n\to+\infty}\P(T_n=\bt)=0.
\end{equation}

Let $\bt \in \mathbb{T}_{0}$ and $x \in \mathcal{L}_0(\bt)$.
We set $D(\bt,x)=M(\bt)-\ind_{\{Z_x(\bt)=1\}}$. Notice that
$D(\bt,x)\leq \Card(\bt)-1$. Elementary computations give for every
$\bt'\in\T_0$ that:
$$\P(\tau=\bt\circledast
\bt')=\frac{1}{p(0)}\P(\tau=\bt)\P(\tau=\bt')\quad \mbox{and}\quad
\P(\tau^*\in\mathbb{T}(\bt,x))=\frac{1}{p(0)}\P(\tau=\bt).$$
As $\tau$ is a.s. finite, we have:
\begin{multline*}
   \P  (\tau \in \mathbb{T}(\bt,x), M(\tau)\in [n,n+n_{0}))\\
\begin{aligned}
&=\sum_{\bt'\in\mathbb{T}_0}\P(\tau=\bt\circledast_x\bt',M(\tau)\in[n,n+n_0))\\
& =\sum_{\bt'\in\mathbb{T}_0}\P(\tau=\bt\circledast_x\bt')\P(M(\bt\circledast_x\bt')\in[n,n+n_0))\\
&=\sum_{\bt' \in \mathbb{T}_{0}}\frac{\P(\tau=\bt)
  \P(\tau=\bt')}{p(0)}\P(\hat M(\bt')+D(\bt,x)\in [n,n+n_{0}))\\
&=\P(\tau^{*} \in \mathbb{T}(\bt,x))\, \P(\hat M(\tau)+D(\bt,x)\in
  [n,n+n_{0})).
\end{aligned}
\end{multline*}
Notice that:
\begin{multline*}
\P(\hat M(\tau)+D(\bt,x)\in
  [n,n+n_{0})) \\
\begin{aligned}
&=\sum_{k=0}^{\Card(\bt)-1}\P(\hat M(\tau)+D(\bt,x)\in
                 [n,n+n_{0}) \mid D(\bt,x)=k)\, \P(D(\bt,x)=k)\\ 
&=\sum_{k=0}^{\Card(\bt)-1}\P(M(\tau)\in [n-k,n+n_{0}-k))\, \P(D(\bt,x)=k).  
\end{aligned}
\end{multline*}
Then we obtain using Assumption \reff{eq:limit} that:
\[
\lim_{n\longrightarrow +\infty}
\frac{\P(\hat M(\tau)+D(\bt,x)\in
  [n,n+n_{0})) }{\P( M(\tau)\in
  [n,n+n_{0})) }=1,
\]
that is
$$\lim_{n\longrightarrow +\infty}\P(\tau \in \mathbb{T}(\bt,x)\mid
M(\tau)\in [n,n+n_{0}))=\P(\tau^{*} \in \mathbb{T}(\bt,x)).$$
This proves the first limit of \reff{eq:conv}.

The second limit is immediate since, for every $n\ge \Card (\bt)$,
$$\P(\tau=\bt\mid M(\tau)\in[n,n+n_0))=0.$$
\end{proof}

The main ingredient for the proof of Theorem \ref{thm:main} is then the
following lemma.

\begin{lemma}\label{lem:ratio}
Let $d$ be the span of the random variable $M(\tau)-1$. We have
\begin{equation}
   \label{eq:limrM}
\lim_{n\rightarrow+\infty } \frac{\P(M(\tau)\in
  [n+1,n+1+d))}{\P(M(\tau)\in [n,n+d))}=1.
\end{equation}
\end{lemma}

\bigskip

The end of this section is  devoted to the proof of Lemma \ref{lem:ratio}, see
Section  \ref{sec:p-ratio}, which  follows  the ideas  of  the proof  of
Theorem 5.1 of \cite{ad14}.

\subsection{Transformation of a subset of a tree onto a tree}

We  recall  Rizzolo's  map  \cite{r15}  which from  $\bt\in  \T_0$  and  a
non-empty  subset  $A$  of  $\bt$   builds  a  tree  $\bt_A$  such  that
$\Card(A)=\Card (\bt_A)$.  We will give a recursive construction of this
map $\phi$:  $(\bt, A)\mapsto \bt_A=\phi(\bt,A)$.  We will  check in
the next  section that this  map is such  that if $\tau$  is a GW  tree then
$\tau_A$ will also be a GW tree for a well chosen subset $A$ of
$\tau$. Figure 1 below shows an example of a tree
$\bt$, a set $A$ and the associated tree $\bt_A$ which helps to
understand the construction.

\medskip

For a vertex $u \in \bt$, recall $C_u(\bt)$ is the set of children of $u$ in
$\bt$. We define for $u \in \bt$:
\[
R_u(\bt)=\bigcup _{w\in \anc(u)} \{v\in C_w(\bt); \, u<v\}
\]
the vertices of $\bt$ which are larger than $u$ for the lexicographic
order and are children of $u$ or of one of its ancestors. 
For a vertex $u \in \bt$, we shall consider $A_u$ the set of elements of $A$ in the fringe
subtree above $u$:
\begin{equation}
   \label{eq:A_u}
A_u=A \cap F_{u}(\bt)=A \cap \{uv ; v\in S_{u}(\bt)\}.
\end{equation}

Let $\bt\in  \T_0$ and  $A\subset \bt$ such  that $A\neq  \emptyset$. We
shall  define  $\bt_A=\phi(\bt, A)$  recursively.   Let  $u_{0}$ be  the
smallest  (for the  lexicographic order)  element of  $A$. Consider  the
fringe subtrees  of $\bt$ that are rooted  at the vertices in  $R_{u_0}(\bt)$ and
contain at least one vertex in $A$, that  is
$(F_u(\bt); u\in R_{u_0}^A(\bt))$, with
\[
R_{u_0}^A(\bt)=\left\{ u\in R_{u_0}(\bt); \, A_u\neq \emptyset \right\}=\left\{ u\in R_{u_0}(\bt); \, \exists v\in A  \text { such that } u \in    \anc(v)\right\}.
\]
Define the number of children of the root of
  tree $\bt_A$ as the number of those fringe subtrees:
\[
k_\emptyset(\bt_A)=
\Card (R_{u_0}^A(\bt)).
\]
If  $k_\emptyset(\bt_A)=0$  set  $\bt_A=\{\emptyset\}$. Otherwise  let  $u_1<
\ldots<   u_{k_\emptyset(\bt_A)}$    be   the   ordered    elements   of
$R_{u_0}^A(\bt)$ with respect to the lexicographic order on $\mathcal{U}
$.  And we define $\bt_A=\phi(\bt, A)$ recursively by:
\begin{equation}
   \label{eq:rec-tA}
F_i(\bt_A)=\phi \left(F_{u_i}(\bt), A_{u_i}\right) \quad  \text{for $1\leq i\leq
k_\emptyset(\bt_A)$.}
\end{equation}
Since $\Card(A_{u_i})<\Card(A)$, we deduce $\bt_A=\phi(\bt, A)$ is well
defined and it is a tree by construction. Furthermore, we clearly have
that $A$ and $\bt_A$ have the same cardinal:
\begin{equation}
   \label{eq:ta=a}
\Card(\bt_A)=\Card(A).
\end{equation}

\begin{center}
\begin{figure}[H]\label{fig:Rizzolo}
\includegraphics[width=14cm]{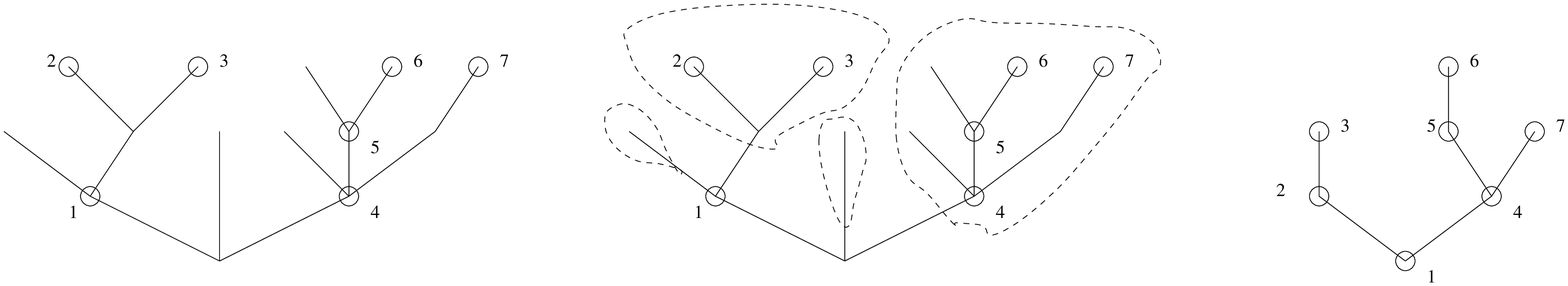}
\caption{Left: A tree $\bt$ and the set $A$.
Center: The fringe subtrees rooted at the vertices in $R_{u_0}(\bt)$.
Left: the tree $\bt_A$.
The labels have no signification, they only show which node of $\bt$
corresponds to a node of $\bt_A$}
\end{figure}
\end{center}

\subsection{Distribution of the number of marked nodes}

Let $(\tau, \cm(\tau))$  be a marked GW  tree with critical
offspring distribution $p$ satisfying \reff{eq:assumption} and mark function
$q$. Recall $\gamma=\P(M(\tau)>0)=\P(\cm(\tau)\neq \emptyset)$.

Let  $((X_{i}, Z_i), i\in \N^*)$  be i.i.d.  random variables  such that
$X_i$ is distributed according to $p$ and $Z_i$ is conditionally on
$X_i$ Bernoulli with parameter $q(X_i)$. We define:
\begin{itemize}
\item $G=\inf \{k\in \mathbb{N}^{*};\ \sum^{k}_{i=1}(X_{i}-1)=-1\}$.
\item $N=\inf \{k\in \mathbb{N}^{*};\ Z_{k}=1\}$.
\item $\tilde{X}$  a random variable  distributed as
  $1+\sum^{N}_{i=1}(X_{i}-1)$ conditionally on  $\{N\leq G\}$.
\item $Y$ a random variable which is conditionally on $\tilde X$
  binomial with parameter $(\tilde X, \gamma)$. 
\end{itemize}

We say that a probability distribution on  $\N$ is
aperiodic if  the span of its support  restricted to $\N^*$ is 1.  
The following result  is immediate as the distribution $p$
of $X_1$ satisfies \reff{eq:assumption}.
\begin{lem}
   \label{lem:LY=ok}
The distribution of $Y$ satisfies \reff{eq:assumption} and if $\gamma<1$
then it is aperiodic.
\end{lem}

Recall that  for a tree $\bt\in \T_0$, we have:
\begin{equation}
   \label{eq:=tree}
\sum_{u\in \bt} (k_u(\bt) -1)=-1
\end{equation}
and  $\sum_{u\in \bt,  u<v} (k_u(\bt)  -1)>-1$ for  any $v\in  \bt$.  We
deduce that $G$  is distributed according to $\Card(\tau)$  and thus $N$
is  distributed like  the index  of the  first marked  vertex along  the
depth-first walk of $\tau$. Then, we have:
\begin{equation}
   \label{eq:N<G}
\gamma=\P(N\leq G). 
\end{equation}

We denote by $(\tau^0,\cm(\tau ^0))$ a random marked tree distributed as
$(\tau,              \cm(\tau))$             conditioned              on
$ \{\mathcal{M}(\tau)\neq \emptyset\}$. By construction, $\Card(\tau^0)$
is distributed as $G$ conditioned on $\{N \leq G \}$.

\begin{lem}
   \label{lem:cardMt}
   Under the hypothesis of this  section, we have that $\tau^0_{\cm(\tau
     ^0)}=\phi(\tau^0, \cm(\tau^0))$ is a critical GW tree with the law
   of $Y$ as offspring distribution.
\end{lem}

\subsection{Proof of Lemma \ref{lem:cardMt}}
\label{sec:p-lem}
In order to simplify notation, we write $\tilde \tau$ for $\tau^0_{\cm(\tau
  ^0)}=\phi(\tau^0, \cm(\tau^0))$ and for $u\in \tau^0$, we set $R_u$
for $R_u(\tau^0)$.

\begin{lem}
   \label{lem:tildet=GW}
The random tree $\tilde \tau$ is a GW tree with offspring distribution
the law of $Y$. 
\end{lem}

\begin{proof}
  Let $u_0$  be the  smallest (for the  lexicographic order)  element of
  $\cm(\tau^0)$.   The  branching property  of  GW  trees implies  that,
  conditionally  given  $u_0$  and  $R_{u_0}$, the  fringe  subtrees  of
  $\tau^{0}$     rooted     at     the    vertices     in     $R_{u_0}$,
  $(S_{u}(\tau^0),  u\in R_{u_0})$  are independent  and distributed  as
  $\tau$.  Recall  notation \reff{eq:A_u}  so  that  the set  of  marked
  vertices    of    the    fringe    subtree   rooted    at    $u$    is
  $\cm_u(\tau^0)=\cm(\tau^0)     \bigcap      F_{u}(\tau^0)$.     Define
  $\tilde \cm_u(\tau^0)=\{v; \, uv\in \cm_u(\tau^0)\}$ the corresponding
  marked vertices  of $S_u(\bt)$.  Then, the  construction of  the marks
  $\cm(\tau)$    implies   that    the   corresponding    marked   trees
  $((S_u(\tau^0), \tilde \cm_u(\tau^0)),  u\in R_{u_0})$ are independent
  and   distributed   as   $(\tau,    \cm(\tau))$.   Notice   that   for
  $u \in  R_{u_0}$, the fringe  subtree $F_u(\tau^0)$ contains  at least
  one mark iff $u$ belongs to
$$ R_{u_0}^{\cm(\tau^0)} =\left\{ u\in R_{u_0};
    \, \exists v\in \cm(\tau^0) \text{  such that  } u \in
    \anc(v)\right\}.$$
  Then by
  considering only the fringe subtrees  containing at least one mark, we
  get  that, conditionally  on $  R_{u_0}^{\cm(\tau^0)} $,  the subtrees
  $((S_u(\tau^0),   \tilde \cm_u(\tau^0)),  u\in   R_{u_0}^{\cm(\tau^0)})$  are
  independent  and distributed  as $(\tau^0,  \cm(\tau^0))$.  We  deduce
  from   the   recursive   construction   of   the   map   $\phi$,   see
  \reff{eq:rec-tA},  that  $\tilde  \tau$  is a  GW  tree.   Notice
  that the
  offspring distribution of  $\tilde \tau$ is given  by the distribution
  of the cardinal of $ R_{u_0}^{\cm(\tau^0)}$.
  We now compute the corresponding offspring distribution. We first give
  an elementary  formula for the  cardinal of $R_u(\bt)$.  Let  $\bt \in
  \T_0$ and $u\in \bt$.  Consider the tree $\bt'=R_u(\bt) \bigcup \{v\in
  \bt; \, v\leq u\}$. Using \reff{eq:=tree} for $\bt'$, we get:
\[
-1=\sum_{v\in \bt'} (k_v(\bt') -1)
= \sum_{v\in \bt; \, v\leq u} (k_v(\bt') -1) + \sum_{v\in R_u(\bt)} (-1).
\]
This gives  $\Card(R_u(\bt))=1+ \sum_{v\in  \bt; \, v\leq  u} (k_v(\bt')
-1)$.  We  deduce from the  definition of  $\tilde X$ that  $\Card(R_{u_0})$ is
distributed as $\tilde  X$.  We deduce from the first  part of the proof
that   conditionally   on    $\Card(R_{u_0})$,   the   distribution   of
$\Card(R_{u_0}^{\cm(\tau^0)})$     is     binomial    with     parameter
$(\Card(R_{u_0}(\tau^0)),  \gamma)$.   This  gives  that  the  offspring
distribution of $\tilde \tau$ is given by the law of $Y$.
\end{proof}

\begin{lemma}\label{lem:critique}
The GW tree $\tilde \tau$ is critical. 
\end{lemma}

\begin{proof}
  Since the  offspring distribution is the  law of $Y$ we  need to check
  that    $\E[Y]=1$ that is $\gamma\E[\tilde X]=1$ since $Y$ is
  conditionally on $\tilde X$ binomial with parameter $(\tilde X,
  \gamma)$. 

  Recall $N$ has finite expectation as $\P(Z_1=1)>0$, is not independent 
  of $(X_{i})_{i\in \mathbb{N}^{*}}$ and is  a stopping time  with respect  
  to the filtration  generated by $((X_i, Z_i), i\in \N ^*)$. 
  Using Wald's equality and
  $\E[X_i]=1$, we get $\E\left[\sum^{N}_{i=1}(X_{i}-1)\right]=0$ and
  thus using the definition of $\tilde X$ as well as \reff{eq:N<G}:
\[
\gamma \E[\tilde X] =\gamma+\E\left[\sum^{N}_{i=1}(X_{i}-1) \ind_{\{N\leq
  G\}}\right]= \gamma-\E\left[\sum^{N}_{i=1}(X_{i}-1) \ind_{\{N>G\}}\right].
\]
We have:
\begin{align*}
\E\left[\sum^{N}_{i=1}(X_{i}-1) \ind_{\{N>G\}}\right]
&=\E\left[\sum^{G}_{i=1}(X_{i}-1) \ind_{\{N>G\}}\right] + \P(N>G)
\E\left[\sum^{N}_{i=1}(X_{i}-1) \right] \\
&= - \P(N>G)\\
&= \gamma-1,
\end{align*}
where we used the  strong Markov property of  $((X_i, Z_i), i\in \N  ^*)$ at the
stopping time $G$ for the first equation, the definition of $T$ and
Wald's equality for the second, and   \reff{eq:N<G} for the third. 
We deduce that $\E[Y]=\gamma \E[\tilde X] =1$, which ends the proof. 
\end{proof}

\subsection{Proof of \reff{eq:limrM}}
\label{sec:p-ratio}

According to Lemma \ref{lem:cardMt} and \reff{eq:ta=a},  we have that  $M(\tau^0)$ is
distributed as the total size of a critical GW whose offspring
distribution satisfies \reff{eq:assumption}. The proof of Proposition
4.3 of \cite{ad14} (see Equation (4.15) in \cite{ad14}) entails that  if $\tau'$
is a critical GW tree, then, if $d$ denotes the span of the random variable
$\Card(\tau')-1$, we have
\[
\lim_{n\rightarrow\infty}\frac{\P(\Card(\tau')\in[n+1,n+1+d))}{\P(\Card(\tau')\in[n,
 n+d))}=1.
\]

\section{Protected nodes}\label{sec:protected}

Recall that a node of a tree $\bt$ is protected if it is not a leaf
and none of its offsprings is a leaf. We denote by $A(\bt)$ the number
of protected nodes of the tree $\bt$.

\begin{theorem}\label{thm:protected}
Let $\tau$ be a critical GW tree with offspring distribution $p$
satisfying \reff{eq:assumption} and let $\tau^*$ be the associated
Kesten's tree. Let $\tau_n$ be a random tree
distributed as $\tau$ conditionally given $\{A(\tau)=n\}$.
Then: $$ \lim_{n\longrightarrow +\infty}\dist(\tau_n)=\dist(\tau^{*}).$$
\end{theorem}

\begin{proof}
Notice that $\P(A(\tau)=n)>0$ for all $n\in \N$. Notice that the
functional $A$ satisfies the additive property of \cite{ad14}, namely
for every $\bt\in\T$, every $x\in\mathcal{L}_0(\bt)$ and every $\bt'\in\T$ that is not reduced to the
root, we have
\begin{equation}\label{eq:additivity}
A(\bt\circledast _x\bt')=A(\bt)+A(\bt')+D(\bt,x)
\end{equation}
where $D(\bt,x)=1$ if $x$ is the  only child of its first ancestor which
is  a  leaf  (therefore  this  ancestor  becomes  a  protected  node  in
$\bt\circledast  _x\bt'$)  and  $D(\bt,x)=0$  otherwise.   According  to
Theorem 3.1 of \cite{ad14}, to end the proof it is enough to check that
\begin{equation}\label{eq:ratio}
\lim_{n\to+\infty}\frac{\P(A(\tau)=n+1)}{\P(A(\tau)=n)}=1.
\end{equation}

For        a       tree        $\bt\neq       \{\emptyset\}$,        let
$\bt_{\mathbb{N}^{*}}=\phi(\bt,\bt\setminus  \cl_0(\bt))$  be  the  tree
obtained from $\bt$ by removing the leaves. Let $\tau^0$ be a random tree 
distributed as $\tau$  conditioned to $\{k_{\emptyset}(\tau)>0\}$. Using 
Theorem 6 and Corollary 2  of   \cite{r15}  with   $A=\N^*$  (or Lemma \ref{lem:cardMt}  with
$q(k)=\ind_{\{k>0\}}$),  we  have   that  $\tau^0_{\mathbb{N}^{*}}$  is  a
critical GW tree with offspring distribution:
\[
p_{\mathbb{N}^{*}}(k)=\sum^{+\infty}_{n=\max(k,1)}
p(n)\, \binom{n}{k}(p(0))^{n-k}(1-p(0))^{k-1}, 
\quad k\in \N.
\]
Conditionally  given  $\{\tau^0_{\N^*}=\bt\}$,   we  consider  independent
random  variables   $(W(u),u\in\bt)$  taking  values  in   $\N^*$  whose
distributions  are  given  for   all  $u\in\bt$  by  $\P(W(u)=0)=0$  for
$k_u(\bt)=0$ and otherwise for $k_u(\bt)+n>0$ (remark that
$p_{\N^*}(k_u(\bt))>0$), by
\[
\P(W(u)=n)=
\frac{p(k_u(\bt)+n)}{p_{\N^*}(k_u(\bt))}
\binom{k_u(\bt)+n}{n}p(0)^n(1-p(0))^{k_u(\bt)-1} . 
\]
In particular for $k_u(\bt)>0$, we have:
\begin{equation}
   \label{eq:W=0}
\P(W(u)=0)=
\frac{p(k_u(\bt))}{p_{\N^*}(k_u(\bt))}
(1-p(0))^{k_u(\bt)-1} . 
\end{equation}

Then, we define a new tree $\hat\tau$ by grafting, on every vertex $u$
of $\tau_{\N^*}^0$, $W(u)$ leaves in a uniform manner, see  Figure
\ref{fig:elagage}.

\begin{figure}[H]
\includegraphics[width=12cm]{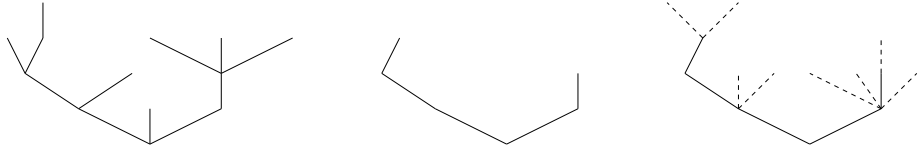}
\caption{The trees $\tau^0$, $\tau^0_{\N^*}$ and $\hat \tau$}\label{fig:elagage}
\end{figure}

More precisely, given $\tau_{\N^*}^0$ and $(W(u),u\in \tau_{\N^*}^0)$,
we define a tree $\hat\tau$ and a random map $\psi:
\tau_{\N^*}^0\longmapsto \hat\tau$ recursively in the following way. We set
$\psi(\emptyset)=\emptyset$. Then, given
$k_\emptyset(\tau_{\N^*}^0)=k$, we set
$k_\emptyset(\hat\tau)=k+W(\emptyset)$. We also consider a family
$(i_1,\ldots,i_{k})$ of integer-valued random variables such that
$(i_1,i_2-i_1,\ldots, i_k-i_{k-1}, W(u)+k+1-i_{k})$ is a uniform
positive partition
of $W(u)+k+1$. Then, for every $j\le k$ such that $j\not\in
\{i_1,\ldots,i_{k}\}$, we set $k_j(\hat\tau)=0$ i.e. these are leaves
of $\hat\tau$. For every $1\le j\le k$, we set $\psi(j)=i_j$ and we
apply to them the same construction as for the root and so on.

\begin{lemma}
The new tree $\hat \tau$ is distributed as the  original tree $\tau^0$.
\end{lemma}

\begin{proof}
Let $\bt \in \mathbb{T}_{0}$. As $\P(\hat \tau=\{\emptyset\})=0$,  we assume that 
$k_{\emptyset}(\bt)>0$. Let $\bt_{\mathbb{N}^{*}}$  be  the  tree
obtained from $\bt$ by removing the leaves. Using \reff{eq:=tree}, we have: 
\begin{align*}
\P(\hat \tau=\bt)
&=\prod_{u \in t_{\mathbb{N}^{*}}}p_{\mathbb{N}^{*}}(k_{u}(\bt_{\mathbb{N}^{*}}))\P(W(u)=k_u(\bt)-k_u(\bt_{\mathbb{N}^{*}}))\frac{1}{\binom{k_u(\bt)}{k_u(\bt)-k_u(\bt_{\mathbb{N}^{*}})}}\\
&=\frac{\P(\tau=t)}{1-p(0)}\\
&=\P(\tau^0=t).
\end{align*}              
\end{proof}

Notice that  the protected nodes  of $\hat \tau$  are exactly the  nodes of
$\tau^0_{\N^*}$ on which we did not add leaves i.e. for which $W(u)=0$. If
we set $\cm(\tau^0_{\N^*})=\{u\in\tau^0_{\N^*},\ W(u)=0\}$, we have $M(\tau^0_{\N^*})=A(\hat \tau)$.

Using \reff{eq:W=0}, we get that the corresponding mark function $q$ is given
by: 
\[
q(k)=\frac{p(k)(1-p(0))^{k-1}}{p_{\mathbb{N}^{*}}(k)}\ind_{\{k\geq
  1\}}.
\]
As $\hat \tau$ is distributed as $\tau^0$, we have:
\[
\lim_{n\to+\infty}\frac{\P(A(\tau^0)=n+1)}{\P(A(\tau^0)=n)}
=\lim_{n\to+\infty}\frac{\P(A(\hat\tau)=n+1)}{\P(A(\hat\tau)=n)}
=\lim_{n\to+\infty}\frac{\P(M(\tau^0_{\N^*})=n+1)}{\P(M(\tau^0_{\N^*})=n)}\cdot
\]
As $\tau^0_{\N^*}$ is a critical GW tree, we deduce from
Lemma \ref{lem:ratio} that
$$\lim_{n\to+\infty}\frac{\P(M(\tau^0_{\N^*})=n+1)}{\P(M(\tau^0_{\N^*})=n)}=1.$$
As $\P(A(\tau)=n)=\P(A(\tau)=n \vert
k_{\emptyset}(\tau)>0)\P(k_{\emptyset}(\tau)>0)$ and 
$\P(A(\tau)=n \vert
k_{\emptyset}(\tau)>0)=\P(A(\tau^0)=n)$
for $n\geq 2$,  
we obtain \reff{eq:ratio} and hence end the proof.
\end{proof}

\subsection*{Acknowledgements}
 
 We would like to thank the referee for his useful comments which helped to improve the paper. 

\bibliographystyle{abbrv}
\bibliography{biblio}

\end{document}